\documentclass[reqno]{amsart}

\usepackage{amsfonts,amsmath,amssymb}
\usepackage{latexsym}

\usepackage[]{hyperref}

\usepackage[latin1]{inputenc}

\allowdisplaybreaks

\numberwithin{equation}{section}

\newtheorem{theorem}{Theorem}[section]
\newtheorem{lemma}[theorem]{Lemma}
\newtheorem{proposition}[theorem]{Proposition}
\newtheorem{corollary}[theorem]{Corollary}

\theoremstyle{definition}
\newtheorem{remark}[theorem]{Remark}

\DeclareMathOperator{\re}{Re}
\DeclareMathOperator{\im}{Im}

\newcommand{\norm}[1]{ \left\|  #1 \right\| }

\def\al{\alpha}

\def\lam{\lambda}

\def\vphi{\varphi}
\def\De{\Delta}

\def\iy{\infty}

\def\cB{\mathcal{B}}
\def\cE{\mathcal{E}}
\def\cH{\mathcal{H}}

\newcommand{\cal}{\mathcal}

\newcommand{\R}{\mathbb{R}}

\def\sideremark#1{\ifvmode\leavevmode\fi\vadjust{\vbox to0pt{\vss
 \hbox to 0pt{\hskip\hsize\hskip1em
\vbox{\hsize2cm\tiny\raggedright\pretolerance10000
 \noindent #1\hfill}\hss}\vbox to8pt{\vfil}\vss}}}%

{\end{list}}

\begin{document}

\title[Damped NLS with Stark effect]{Damped nonlinear Schr\"odinger equation with Stark effect}

\author{Yi Hu} 
\address{Department of Mathematical Sciences,
Georgia Southern University,
Statesboro, GA 30458}
\email{yihu@georgiasouthern.edu}

\author{Yongki Lee}
\email{yongkilee@georgiasouthern.edu}

\author{Shijun Zheng}
\email{szheng@georgiasouthern.edu}

\begin{abstract} We study the $L^2$-critical damped NLS with a Stark potential.  
We prove that the threshold for global existence and finite time blowup of this equation is given by $\norm{Q}_2$,
where $Q$ is the unique positive radial  solution of $\De Q+ |Q|^{4/d}Q=Q$ in $H^1(\R^d)$. 
Moreover,  
in any small neighborhood of $Q$, 
there exists an initial data $u_0$ above the ground state such that the solution flow  admits the log-log blowup speed. 
This verifies the structural stability for the ``$\log$-$\log$ law'' associated to the NLS mechanism under the perturbation by a damping term and a Stark potential. 
The proof of our  main theorem is based on the Avron-Herbst formula and the
analogous result for the unperturbed damped NLS.  
\end{abstract}


\maketitle

\section{Introduction}

Consider the damped nonlinear Schr\"odinger equation (dNLS) with Stark effect:
	\begin{align}\label{eq:dnls_stark}
	i u_t = - \Delta u + (E \cdot x) u - |u|^{p - 1} u - i a u, \qquad
	u(0) = u_0\in \cH^1. 
	\end{align}
Here $p\in (1,1+\frac{4}{n-2})$, $n\ge 1$, $u = u(t, x): \mathbb{R}^{1 + n} \rightarrow \mathbb{C}$ is the wave function,
$V(x)=V_E(x):=E \cdot x$ is a Stark potential with $E \in \mathbb{R}^n \setminus \{ 0 \}$,
and $i a u$ is a linear damping to the system with $a > 0$ being the coefficient of friction. 
The initial data $u_0$ is in the energy space 
$\mathcal{H}^1 := \left\{ \phi \in H^1: x \phi \in L^2 \right\}$,
whose   norm is given by
$\| \phi \|_{\mathcal{H}^1} := \left(\| \phi \|_{H^1}^2 + \| x \phi \|^2_{L^2}\right)^{1/2}$,  
where $H^1$  denotes the usual Sobolev space.  The associated energy of the dNLS is given by
\begin{align}
	\mathcal{E}_V(u)
	:= \int_{\mathbb{R}^n} \left(|\nabla u|^2 
	+  V_E |u|^2 
	- \frac{2}{p + 1}  |u|^{p + 1} \right) dx. \label{eq:stark_energy}
	\end{align}
In physics,
$E$ may represent an electric field or quantum gravity \cite{CyFroeKirSi, TsuWa},
and damped NLS appear in  nonlinear optics,
plasma physics,
and fluid mechanics \cite{
Fib01damped, FKKV19, KhKMTh10,PePoVaz95}. 

The hamiltonian $H = H_E: = -\Delta + E \cdot x$ arises in the study of hydrogen model
in connection with the resonance phenomenon associated with quasi-stationary states \cite{GraGre78, Ha07}.  
The operator $H_E$ is essentially selfadjoint in $C_0^\infty(\mathbb{R}^3)$,
which has absolutely continuous spectrum $\mathbb{R} = (-\infty, \infty)$.
The presence of a Stark potential has an resonance effect such that it shifts the discrete energies of the hydrogen atom into resonances
(pseudo-eigenvalues)
via the hamiltonian $H_{Z, E} = - \De - \frac{Z}{|x|} + E \cdot x$,
where $Z \ge 0$ is the atomic number.
This distinguishes from the harmonic oscillator  $- \De + |x|^2$ 
in that $H_{Z,E}$  has weaker decaying ``bound state'',
and so $V_E= E \cdot x$  is also called the spatially damped oscillator.

On the other hand, when $a > 0$, the linear term $-ia u$ is present as a temporal damping effect for the NLS equation.
So it might be of analytical  interest to study the Schr\"odinger type system as  (\ref{eq:dnls_stark}). 
There have been a large literature in the field of wave dispersion-dissipation in physics and numerics 
\cite{AkDouKaMcKi,BeDiaz23, FKKV19,Lau95,Malomed22mu}. 

If $E = 0$ and $a = 0$,
then equation \eqref{eq:dnls_stark} becomes the classic NLS
	\begin{align}\label{eq:nls}
	i u_t = - \Delta u - |u|^{p - 1} u\,, \qquad
	u(0) = u_0\in H^1
	\end{align} 
and its well-posedness and blowup have been studied extensively in the $H^1$-subcritical and $H^1$-critical cases $p \leq 1 + \frac{4}{n - 2}$, see 
\cite{Ca,Tao}.
 The following quantities are conserved in time for (\ref{eq:nls}):
	\begin{alignat}{9}
	&\text{(Mass)} & \mathcal{M}(u) &:= \| u(t) \|^2_{L^2} \label{eq:mass} \\
	&\text{(Energy)}	& \mathcal{E}_0(u) &:= \| \nabla u(t) \|_{L^2}^2 - \frac{2}{p + 1} \| u(t) \|_{L^{p + 1}}^{p + 1} \label{eq:energy} \\
	&\text{(Momentum)} & \qquad \mathcal{P}(u) &:= \im \left( \int_{\mathbb{R}^n} \overline{u} \nabla u dx \right)\,. \label{eq:momentum}
	\end{alignat}
	
In the $L^2$-critical regime $p = 1 + \frac{4}{n}$,
if we let $Q$ be the unique positive and radial solution in $H^1$ to the elliptic equation 
	\begin{align}\label{eQ0}
	 \Delta Q + |Q|^\frac{4}{n} Q=Q\,,
	\end{align}
then Weinstein \cite{Wein} showed that $\| Q \|_{L^2}$ is the threshold 
for global existence and finite time blowup  of the Cauchy problem \eqref{eq:nls} in $H^1$.
  Namely, if
 $\| u_0 \|_{L^2} < \| Q \|_{L^2}$,
then the solution $u(t)$ of (\ref{eq:nls}) is global in $H^1$,
while for any $c \geq \| Q \|_{L^2}$,
there exists  $u_0 \in H^1$ with $\| u_0 \|_{L^2} = c$ such that the solution $u(t)$ blows up in finite time $T^*>0$.  
Further, Merle \cite{Mer93} showed that the set of all minimal mass blowup solutions at the ground state level $\norm{Q}_2$ consists of 
pseudo-conformal transforms of the solitary wave $e^{it}Q(x)$: 
\begin{align} 
& S(t,x)= \frac{e^{i\theta}}{|T-t|^{d/2}} Q(\frac{x-x_0}{T-t}) 
e^{-i\frac{|x-x_0|^2}{4 (T-t )}} e^{\frac{i}{T-t}}\,        \label{pSeudo}
\end{align} 
where $(\theta, T, x_0 )\in \R\times \R\times \R^n$ are parameters.  

When the initial data is above the ground state level,  
under the assumption of certain spectral property (see Remark \ref{re:spectral}),  
Merle and Rapha\"el
(\cite{MerRa0, MerRa1})
proved the sharp blowup speed for the solutions,
i.e.,
there exists a small universal constant $\alpha^* > 0$ such that for all $u_0 \in \cB_{\alpha^*}$ with negative energy $\mathcal{E}_0(u_0)<0$, 
the solution $u(t)$ blows up as $t \rightarrow T^*$ with the speed
(called ``log-log law'')
	\begin{align}\label{eq:loglog_speed}
	\| \nabla u(t) \|_{L^2}
	\approx \left( \frac{\log \left| \log \left( T^* - t \right) \right|}{T^* - t} \right)^\frac{1}{2},
	\end{align}
where
	\begin{align*}
	\cB_{\alpha}
	:= \left\{ \phi \in H^1:
	\| Q \|_{L^2}
	< \| \phi \|_{L^2}
	< \| Q \|_{L^2} + \alpha \right\}.
	\end{align*}
This log-log regime is also known to be stable in $H^1$  \cite{Ra05stabi}.

If $E = 0$ in equation \eqref{eq:dnls_stark}, then one has the damped NLS
	\begin{align}\label{eq:dnls}
	i \varphi_t = - \Delta \varphi - |\varphi|^{p - 1} \varphi - i a \varphi, \qquad
	\varphi(0) = \vphi_0=u_0\in H^1\,.
	\end{align}
The local well-posedness  in $H^1$ for the dNLS is well-known (see e.g. \cite{Ca,Tsu}).
Precisely speaking,
for every $u_0 \in H^1$, 
there exists $T>0$ and a unique solution $\varphi \in C \left( [0, T), H^1 \right)$ 
of the Cauchy problem \eqref{eq:dnls},
where $[0,T)=[0,T_{max})$ is the maximal time interval of existence. 
Moreover,  if $T_{max} $ is finite,   
then $\| \nabla \varphi(t) \|_{2} \rightarrow \infty$ as $t \rightarrow T_{max}$.
Because of the damping term $-i a \varphi$,
the dynamics of the solution may behave differently from that of the classic NLS.
For instance,
the mass,
energy,
and momentum \eqref{eq:mass}--\eqref{eq:momentum} are not conserved.
Darwich \cite{Dar} studied the Cauchy problem \eqref{eq:dnls}
in the $L^2$-critical regime $p = 1 + \frac{4}{n}$ when $n = 1, 2, 3, 4$,
and he proved that $\| Q \|_{2}$ is the sharp threshold for the blowup phenomenon in $H^1$.
Furthermore,  there exists $\alpha_0>0$ such that for all $a>0$ and arbitrary $\al\in (0,\al_0)$,    
 there exists a blow-up solution in the log-log regime corresponding to some $u_0 \in \cB_{\alpha}$,     
see \mbox{Theorem \ref{thm:dnls_threshold}.} 

Motivated by the abovementioned work, 
in this paper we concentrate on the $L^2$-critical case $p=1+4/n$ for the focusing dNLS (\ref{eq:dnls_stark}) with a Stark potential. 
We shall  establish the sharp threshold for global existence and the log-log law for 
for the Cauchy problem \eqref{eq:dnls_stark} in the weighted Sobolev space 
$\mathcal{H}^1$.  
Our  main result is stated as follows. 

\begin{theorem}\label{thm:dnls_stark_threshold} Let $p = 1 + \frac{4}{n}$ and $a>0$. Suppose
$u_0 \in \mathcal{H}^1(\mathbb{R}^n)$ for $1 \leq n \leq 4$. 
	\begin{enumerate}
	\item If $\| u_0 \|_{L^2} \leq \| Q \|_{L^2}$,
	then the solution of \eqref{eq:dnls_stark} is global in time such that $u(t)\in C([0,\iy),\cH^1)$.
	
	\item There exists a small $\alpha_0 > 0$ such that for arbitrary $\al \in (0, \alpha_0)$,
	there exists  $u_0 \in \cB_\al\cap \mathcal{H}^1$ 
	such that the corresponding solution $u(t)$ of (\ref{eq:dnls_stark}) blows up at $T^* < \infty$ with the log-log speed \eqref{eq:loglog_speed}.
	\end{enumerate}
\end{theorem}

The proof of Theorem \ref{thm:dnls_stark_threshold}
(see Section \ref{sect:proof_main_theorem})
relies on the Avron-Herbst transform
(Proposition \ref{prop:ah}) and Theorem \ref{thm:dnls_threshold}.
An alternative, independent proof of part (a) in all dimensions is given as a corollary
(Corollary \ref{cor:proof_a}) 
of a limiting profile result in Theorem \ref{thm:limiting_profile}.

We wish to mention that for the  $L^2$-critical NLS \eqref{eq:nls},
there exist minimal mass blowup solutions (\ref{pSeudo}) with $\| u_0 \|_{2} = \| Q \|_{2}$ 
and   
  pseudo-conformal blow-up  speed \mbox{$(T^* - t)^{-1}$.} 
However, in  the presence of damping,
$\| u_0 \|_{2} = \| Q \|_{2}$ will lead to a global solution to the equation \eqref{eq:dnls_stark},
so  Theorem \ref{thm:dnls_stark_threshold} also proves the {\em none existence of minimal mass blow-up solutions} for (\ref{eq:dnls_stark}).

 {Theorem \ref{thm:dnls_stark_threshold} shows the structural stability for the NLS mechanism:  The log-log law continues to hold when the free NLS
  $i u_t=-\De u-|u|^{4/n} u$ is perturbed with a damping term and a linearly growth potential. }  
  Similar results  have been  obtained for other Schr\"odinger type equations rencently 
\cite{BaHuZh2019a,FanSuZh22stochastic,PlanRa07,SunZh20}. 
However, from the proofs in either \cite{Dar} or \cite{DarwMolin16},  it is not evident whether 
such $\log$-$\log$ regime is topologically stable for the dissipative NLS (\ref{eq:dnls_stark}) in an electric field.   

Our second main result is concerned with finding a sufficient condition 
 on singular (blow-up) solutions for (\ref{eq:dnls_stark}).  
By the $A$-$H$ transform in Proposition \ref{prop:ah},
 we are able to show that, given any $u_0$ with $\cE_0(u_0)<0$,  
there exists $a_*=a_*(u_0)>0$ such that  
the solution $u(t)$ blows up for all $0<a<a_0$   
 by an application of \cite[Theorem 1.2]{Dinh21b}, 
see Theorem \ref{t:dinh:E0<0}.

\begin{theorem}\label{a0:blup:stark}  Let $p=1+4/n$, $n\ge 1$.  
Suppose $u_0\in \cH^1$ and $\cE_V(u_0)<{\int E\cdot x |u_0|^2.}$   
 Then, there exists $a_*=a_*(\norm{u_0}_{H^1})>0$ such that  for all $a\in (0,a_*)$,     
the corresponding solution     
$u(t)$ of (\ref{eq:dnls_stark})  in $C( [0,T^*), \cH^1 )$     
 blows up finite time on $[0,T^*)$ in the sense that $\norm{\nabla u(t)}_2\to \iy$ as $t\to T^*$.       
\end{theorem}   

Note that this theorem, along with part (a) in Theorem \ref{thm:dnls_stark_threshold}, 
 implies the blow-up  for any initial data 
 in the open region $\{ u\in \cal{B}_{\al_0}\cap \cH^1:  \mathcal{E}_0(u)<0 \}$
 for any prescribed positive $\al_0$.

On the other hand, the numerical result in \cite{Fib01damped} and also \cite{CuiHS17,Dinh21b} 
suggest that increasing the value of $a>a_*$ can give rise to the effect that {\em damping arrests self-focusing}
so that the solution will survive over  infinity time. 
However,  such $a_*$ is dependent on the initial data. 
 Meanwhile, in the case $V_E=0$, a simple scaling argument 
informs that given any solution $u$ of $NLS_a$ in (\ref{a:dnls}), 
then  
  $u_\lam(t,x):=\frac1{\lam^{n/2}} u(\frac{t}{\lam^2},\frac{x}{\lam})$
solves $NLS_{a/\lam^2}$, cf. \cite[Remark 3.2]{Dar}.  
Thus, the problem of stability/instability of blow-up for the dNLS can be notably subtle and sensitive.

We would like to conclude that the results in this article inform that the Stark potential $V_E=E\cdot x$ 
does not seem to essentially change the blow-up dynamics by its interactions with the damping term and the $L^2$-critical nonlinearity. 
However, this potential may affect the scattering behavior owing to its interaction with a linear potential potential like $|x|^{-\gamma}$, 
as was observed in Ozawa's work and \cite{CarNa}. 
It would be of interest to further study the effect of $V_E$ on the long time asymptotic behavior for equation (\ref{eq:dnls_stark}) when $p<1+4/n$.

\section{Preliminaries}
For the dNLS (\ref{eq:dnls_stark}) with a Stark potential $V_E$, 
the local well-posedness  in $\cH^1$ holds. In fact, following a standard fixed point argument as in  \cite{Ca} or \cite{CarNa,Z2012a}, we
can easily show the following l.w.p. for (\ref{eq:dnls_stark}).
\begin{proposition}\label{p:dnls:VE} Let $p\in (1,1+4/(n-2))$. Suppose  $u_0 \in \cH^1$.  
Then there exists $T:=T^*\in (0,\iy]$ and a unique solution $u(t)$ in $C \left( [0, T), \cH^1 \right)$  
of the Cauchy problem \eqref{eq:dnls}, 
where $[0,T)$ is the lifespan for forward time. 
The blow-up alternative holds:   If $T $ is finite,   
then $\| \nabla u(t) \|_{2} \rightarrow \infty$ as $t \rightarrow T$.
\end{proposition}

We omit the detailed proof of the proposition, but instead  provide 
below  a description of the ``modified conservation laws''
 on the interval of existence $[0, T)$  
for the mass, energy, and momentum of the system.

\begin{proposition}
Let $u$ be a solution of the Cauchy problem \eqref{eq:dnls_stark} on $[0, T)$.
Let $\mathcal{M}(u)$,
$\mathcal{E}_0(u)$,
and $\mathcal{P}(u)$ be defined as in \eqref{eq:mass}--\eqref{eq:momentum},
respectively,
and let $\mathcal{E}_V(u)$ be the associated energy  (\ref{eq:stark_energy}). 
Then
	\begin{align}
	\mathcal{M}(u)
	&= e^{- a t} \mathcal{M}(u_0) \label{eq:mass_decay} \\
	\frac{d}{dt} \mathcal{E}_0(u)
	&= - 2 i \int_{\mathbb{R}^n} E \cdot u \nabla \overline{u} dx
	- 2 a \| \nabla u(t) \|_{L^2}^2
	+ 2 a \| u(t) \|_{L^{p + 1}}^{p + 1} \label{eq:energy_decay} \\
	\frac{d}{dt} \mathcal{E}_V(u)
	&= -2 a \int_{\mathbb{R}^n} E \cdot x |u|^2 dx
	- 2 a \| \nabla u(t) \|_{L^2}^2
	+ 2 a \| u(t) \|_{L^{p + 1}}^{p + 1} \label{eq:stark_energy_decay} \\
	\mathcal{P}(u)
	&= e^{-2at} \left( - tE \mathcal{M}(u_0)^2  + \mathcal{P}(u_0) \right). \label{eq:momentum_decay}
	\end{align}
\end{proposition}

\begin{proof} Equations (\ref{eq:mass_decay}) to (\ref{eq:momentum_decay}) 
 can be verified by straightforward calculations.
For instance,
to verify the identity \eqref{eq:mass_decay},
we multiply both sides of equation \eqref{eq:dnls_stark} by $\overline{u}$
and integrate them with respect to $x$ to obtain
	\begin{align*}
	\int_{\mathbb{R}^n} i u_t \overline{u} dx
	&= \int_{\mathbb{R}^n} \left( - \Delta u + E \cdot x u - |u|^{p - 1} u	- i a u \right) \overline{u} dx \\
	&= \int_{\mathbb{R}^n} \left( |\nabla u|^2 + E \cdot x |u|^2 - |u|^{p + 1} - i a |u|^2 \right) dx.
	\end{align*}
Since
	\begin{align*}
	\frac{d}{dt} \int_{\mathbb{R}^n} |u|^2 dx
	= 2 \re \int u_t \overline{u} dx
	= 2 \im \int i u_t \overline{u} dx
	= - 2 a \int_{\mathbb{R}^n} |u|^2 dx,
	\end{align*}
we obtain the o.d.e. 
	\begin{align*}
	\frac{d}{dt} \| u(t) \|_{L^2}^2
	= - 2 a \| u(t) \|_{L^2}^2, \qquad
	\| u(0) \|_{L^2}
	= \| u_0 \|_{L^2},
	\end{align*}
whose solution yields  \eqref{eq:mass_decay}.

Similarly,
equations \eqref{eq:energy_decay} and \eqref{eq:stark_energy_decay} follow from a direct calculation,
and equation \eqref{eq:momentum_decay} follows via solving the Cauchy problem
	\begin{align*}
	\frac{d}{dt} \mathcal{P}(u) = - E \mathcal{M}(u)^2 - 2 a \mathcal{P}(u), \qquad
	\mathcal{P}(u(0)) = \mathcal{P}(u_0).
	\end{align*}
\end{proof}

\begin{remark}\label{re:mass_decay}
If the system is damping-free,
i.e. $a = 0$,
then the quantities $\mathcal{M}(u)$ and $\mathcal{E}_V(u)$ are conserved.
Also,
from  \eqref{eq:mass_decay}--\eqref{eq:momentum_decay},
it seems that the Stark potential affects the dynamics of all but $\mathcal{M}(u)$.
Indeed,
equation \eqref{eq:mass_decay} also holds when $E = 0$
(which will be used in the proof of Theorem \ref{thm:dnls_stark_threshold} later).
Moreover,
equation \eqref{eq:mass_decay} indicates the nonexistence of solitary waves of the form $u(t, x) = e^{it} \phi(x)$,
for otherwise $\| u(t) \|_{L^2} = \| \phi \|_{L^2}$ does not decay in time.
\end{remark}

The following theorem is the main result in \cite{Dar},
and it will be applied in the proof of Theorem \ref{thm:dnls_stark_threshold}.

\begin{theorem}[Darwich \cite{Dar}]\label{thm:dnls_threshold}
Let $p=1+\frac{4}{n}$ and  $u_0 \in H^1(\mathbb{R}^n)$,  $n = 1, 2, 3, 4$. 
	\begin{enumerate}
	\item If $\| u_0 \|_{L^2} \leq \| Q \|_{L^2}$,
	then the solution to equation \eqref{eq:dnls} is global in $H^1$.
	
	\item There exists a $\delta_0 > 0$ such that,
	for all $a > 0$ and $\delta \in (0, \delta_0)$,
	there exists a $u_0 \in H^1$ with $\| u_0 \|_{L^2} = \| Q \|_{L^2} + \delta$,
	such that the solution to equation \eqref{eq:dnls} blows up in finite time in the log-log regime (\ref{eq:loglog_speed}).
	\end{enumerate}
\end{theorem}  

The proof of Theorem \ref{thm:dnls_threshold} is a modification of the approach in \cite{MerRa04boot} to \cite{MeRa06sharp}. 
The initial  ansatz  for the profile near $t\to T^*$ is given as
\begin{align*} 
& \vphi(t,x)= \frac{e^{i\theta(t)}}{\lam(t)^{d/2}} \left(Q_{b(t)}+ \eta\right) (t, \frac{x-y(t)}{\lam(t) } )  \, 
\end{align*} 
for some geometrical parameters $(b, \lam, y, \theta)=(b(t),\lam(t),y(t),\theta(t))\in \R_+\times \R_+\times \R^n\times \R$
with $\lam(t)\sim 1/\norm{\nabla u(t)}_2$.   These profiles $Q_b$ are regularization of the 
self-similar solutions of (\ref{eq:nls}) that obeys the elliptic equation 
\begin{align*}
  \Delta Q_b  + ib \left( \frac{n}{2}  +x\cdot\nabla \right) Q_b  + |Q_b|^\frac{4}{n} Q_b=Q_b\,.
	\end{align*} 
Thus  $Q_b$ are resultantly  suitable deformation of $Q$ up to some degeneracy of the problem (\ref{eQ0}). 
The geometrical parameters here
 are uniquely defined per some orthogonality conditions in the {\em Spectral Property} \cite[p.164]{MerRa0},
 or \cite{BaHuZh2019a,YangRouZh2018}.    
 
\begin{remark} 
  The limit on the dimension $n\le 4$ is required in  an interpolation inequality in 
the proof of \cite[Lemma 6.2]{Dar},  see also similar proof for \cite[Lemma 4.2]{DarwMolin16}.   
 \end{remark}
 
\begin{remark}\label{re:spectral} The {\em Spectral Property } has not been analytically verified for all dimensions.
In one dimension it was  proved by Merle and Rapha\"el \cite{MerRa0} using the explicit expression of $Q$ to (\ref{eQ0}).  
 The  recent progress in higher dimensions is attributed to \cite{FibMeRa06}, and  
 \cite{YangRouZh2018}, 
where is given an improved numerically-assisted proof for $n \leq 10$ and also for $n = 11, 12$ in the radial case. 
See also the discussions on the rotational NLS in \cite{BaHuZh2019a,BHHZ23t},
where the spectral property is required.  
\end{remark} 

For the proof of Theorem \ref{a0:blup:stark} we need the analogous result in \cite{Dinh21b}, 
where some blow-up conditions for (\ref{eq:dnls}) were obtained. 
Let $J(t):= \int |x|^2 |\vphi|^2$ be the variance. 
In the $L^2$-critical regime $p = 1 + \frac{4}{n}$, the virial identity for the free NLS (\ref{eq:nls}) reads
$\frac{d^2}{dt^2} J(t)= 8 \cE_0(u_0)$, which  can be used to show that 
$u_0\mapsto u(t)$ is a blow-up solution of (\ref{eq:nls}) if $\cE_0(u_0)<0$.  
This result was extended for the dNLS (\ref{a:dnls}) in the absence of $V_E$ in \cite[Theorem 1.2]{Dinh21b} 
\begin{align}\label{a:dnls}
	i \varphi_t = - \Delta \varphi - |\varphi|^{4/n} \varphi - i a \varphi, \qquad
	\varphi(0) = \vphi_0\,.
	\end{align}   
The proof is based on a localized virial identity for the dNLS. Denote \\
\mbox{$\Sigma:=H^1\cap \{u\in L^2: \int |x|^2 |u|^2<\iy\}$.} 

\begin{theorem}[Dinh \cite{Dinh21b}]\label{t:dinh:E0<0}  Let $p=1+4/n$, $n\ge 1$.  
Suppose $u_0\in \Sigma$ and $\cE_0(\vphi_0)<0$.   
 Then, there exists a positive $a_*=a_*(\norm{u_0}_{H^1})$ such that  for all $a\in (0,a_*)$,    
the corresponding solution    
$u(t)$ of (\ref{a:dnls})  in $C( [0,T^*), \Sigma )$   
 blows up finite time on $[0,T^*)$. 
\end{theorem}  

The proof of Theorem \ref{a0:blup:stark} is based on  
  a simple application of the Avron-Herbst formula (\ref{eq:ah_inv}) 
and Theorem \ref{t:dinh:E0<0}.  
  The $A$-$H$ transform allows us to 
   convert solutions  $u(t)$ of (\ref{eq:dnls_stark}) into solutions $\vphi$ of (\ref{a:dnls}). 
We leave the proof as an easy exercise for the reader.

\section{Avron-Herbst Formula and Proof of Theorem \ref{thm:dnls_stark_threshold}}\label{sect:proof_main_theorem}

First we introduce the Avron-Herbst formula, as is well-known
  \cite{AvHerbst, CarNa, CyFroeKirSi}.

\begin{proposition}\label{prop:ah}
Let $T \in (0, \infty]$.
If $\varphi$ is the solution to the Cauchy problem \eqref{eq:dnls} on $[0, T)$,
then for $E \in \mathbb{R}^n \setminus \{ 0 \}$,
the function
	\begin{align}\label{eq:ah}
	u(t, x)
	:= \varphi \left( t, x + t^2 E \right) e^{- i \left( t E \cdot x + \frac{|E|^2 t^3}{3} \right)}
	\end{align}
is the solution to the Cauchy problem \eqref{eq:dnls_stark} on $[0, T)$.

Conversely,
if $u$ is the solution to the Cauchy problem \eqref{eq:dnls_stark} on $[0, T)$,
then the function
	\begin{align}\label{eq:ah_inv}
	\varphi(t, x)
	:= u \left( t, x - t^2 E \right) e^{i \left( t E \cdot x - \frac{2 |E|^2 t^3}{3} \right)}
	\end{align}
is the solution to the Cauchy problem \eqref{eq:dnls} on $[0, T)$.
\end{proposition}

\begin{proof}
Both can be verified by direct computation.
For example,
to verify that $u$ in \eqref{eq:ah} solves the  problem \eqref{eq:dnls_stark},
obviously we have $u(0) = \varphi(0) = u_0$,
and
	\begin{align*}
	u_t (t, x)
	&= \left[ \varphi_t \left( t, x + t^2 E \right)
	+ 2 t E \cdot \nabla \varphi \left( t, x + t^2 E \right) \right]
	e^{- i \left( t E \cdot x + \frac{|E|^2 t^3}{3} \right)} \\
	&\qquad + \varphi \left( t, x + t^2 E \right)
	e^{- i \left( t E \cdot x + \frac{|E|^2 t^3}{3} \right)}
	\left[ - i \left( E \cdot x + |E|^2 t^2 \right) \right] \\
	&= \left[ \varphi_t \left( t, x + t^2 E \right)
	+ 2 t E \cdot \nabla \varphi \left( t, x + t^2 E \right)
	- i E \cdot x \varphi \left( t, x + t^2 E \right) \right. \\
	&\qquad \left. - i |E|^2 t^2 \varphi \left( t, x + t^2 E \right) \right]
	e^{- i \left( t E \cdot x + \frac{|E|^2 t^3}{3} \right)}
	\end{align*}
and
	\begin{align*}
	\Delta u(t, x)
	&= \Delta \left[ \varphi \left( t, x + t^2 E \right) \right]
	e^{- i \left( t E \cdot x + \frac{|E|^2 t^3}{3} \right)}
	+ 2 \nabla \left[ \varphi \left( t, x + t^2 E \right) \right]
	\cdot \nabla \left[ e^{- i \left( t E \cdot x + \frac{|E|^2 t^3}{3} \right)} \right] \\
	&\qquad + \varphi \left( t, x + t^2 E \right)
	\Delta \left[ e^{- i \left( t E \cdot x + \frac{|E|^2 t^3}{3} \right)} \right] \\
	&= \left[ \Delta \varphi \left( t, x + t^2 E \right)
	- 2 i t E \cdot \nabla \varphi \left( t, x + t^2 E \right)
	- t^2 |E|^2 \varphi \left( t, x + t^2 E \right) \right]
	e^{- i \left( t E \cdot x + \frac{|E|^2 t^3}{3} \right)}.
	\end{align*}
Then equation \eqref{eq:dnls_stark} holds if we bring $u$,
$u_t$,
and $\Delta u$ into both sides and use equation \eqref{eq:dnls} for $\varphi$.
\end{proof}

Now we prove Theorem \ref{thm:dnls_stark_threshold} using the Avron-Herbst transform (\ref{eq:ah})-(\ref{eq:ah_inv}).

\begin{proof}[Proof of Theorem \ref{thm:dnls_stark_threshold}]
To prove (a),
note that if $\| u_0 \|_{L^2} \leq \| Q \|_{L^2}$,
then according to Theorem \ref{thm:dnls_threshold},
there exists a global solution $\varphi$ to the Cauchy problem \eqref{eq:dnls} with $\varphi(0) = u_0$.
Applying formula \eqref{eq:ah} to $\varphi$,
we obtain a global solution $u$ to the  problem \eqref{eq:dnls_stark}.

To prove (b),
let $\alpha_0 := \delta_0$ in Theorem \ref{thm:dnls_threshold}.
Then for each $\al \in (0, \alpha_0)$,
there exists an initial value $\| u_0 \|_{L^2} = \| Q \|_{L^2} + \al$
with which the solution $\varphi$ to the Cauchy problem \eqref{eq:dnls}
blows up in finite time $T^*$ at the log-log speed \eqref{eq:loglog_speed}.
Then we obtain $u$ in terms of $\varphi$ on $[0, T^*)$ through formula \eqref{eq:ah},
and $[0, T^*)$ is the maximal time interval for $u$
because otherwise we can use formula \eqref{eq:ah_inv} to extend $\varphi$ beyond $T^*$.
To show that $u$ blows up in the log-log regime,
we have
	\begin{align*}
	\nabla u(t, x)
	&= \nabla \varphi \left( t, x + t^2 E \right) e^{- i \left( t E \cdot x + \frac{|E|^2 t^3}{3} \right)}
	+ \varphi \left( t, x + t^2 E \right) e^{- i \left( t E \cdot x + \frac{|E|^2 t^3}{3} \right)} (- i t E ) \\
	&=: \textup{I} + \textup{II}.
	\end{align*}
It is easy to see that
	\begin{align*}
	\| \textup{I} \|_{L^2}
	= \| \nabla \varphi(t) \|_{L^2}
	\approx \left( \frac{\log| \log( T^* - t) |}{T^* - t} \right)^{\frac{1}{2}} \qquad
	\text{as } \ t \rightarrow T^*.
	\end{align*}
Also,
by the identiy \eqref{eq:mass_decay}
(and Remark \ref{re:mass_decay}),
we have
	\begin{align*}
	\| \textup{II} \|_{L^2}
	= t |E| \left\| \varphi \left( t \right) \right\|_{L^2}
	= t e^{-a t} |E| \left\| u_0 \right\|_{L^2},
	\end{align*}
which is a bounded function on $[0, \infty)$ with maximum occurring at $t = \frac{1}{a}$.
Combining both estimates,
we have
	\begin{align*}
	\| \nabla u(t) \|_{L^2}
	\approx \left\| \text{I} \right\|_{L^2} + \left\| \text{II} \right\|_{L^2}
	\approx \left( \frac{\log| \log( T^* - t) |}{T^* - t} \right)^{\frac{1}{2}} \qquad
	\text{as } \ t \rightarrow T^*.
	\end{align*}
\end{proof}

\begin{remark}\label{re:minimal_mass_remark} 
We have noted in the introduction section that,  
in the presence of damping, 
Theorem \ref{eq:dnls_stark} (a) indicates that there 
 always exists a global solution of \eqref{eq:dnls_stark} when $\| u_0 \|_{2} = \| Q \|_{2}$.  
Indeed,  we will give another proof of this property in the next section as a corollary of a limiting profile result
(Theorem \ref{thm:limiting_profile}), 
which basically says that the mass of a blowup solution will concentrate and will be no less than $\| Q \|_{2}$.
This, 
together with the mass decay \eqref{eq:mass_decay},
will explain why there are no blowup solutions at the level $\| Q \|_{2}$, 
see Corollary \ref{cor:proof_a}.
\end{remark}

\section{A Limiting Profile Result for Equation (\ref{eq:dnls_stark}) }

In this section,
we will first prove a limiting profile result about the mass concentration of blow-up solutions of equation (\ref{eq:dnls_stark}) in
Theorem \ref{thm:limiting_profile}, 
and then give an alternative proof of Theorem \ref{thm:dnls_stark_threshold} (a) for all dimensions.
Throughout this section, 
we assume $p = 1 + \frac{4}{n}$.

In \cite{HmiKe},
Hmidi and Keraani proved a refined version of compactness lemma adapted to the analysis of the blow-up phenomenon
for the Cauchy problem \eqref{eq:nls} in the $L^2$-critical case.
Their result is as follows.

\begin{theorem}[\cite{HmiKe}]\label{thm:compactness_lemma}
Let $\{ v_k \}$ be a bounded sequence in $H^1(\mathbb{R}^n)$ such that
	\begin{align}\label{eq:upper_bound}
	\limsup_{k \rightarrow \infty} \| \nabla v_k \|_{L^2} \leq M
	\end{align}
and
	\begin{align}\label{eq:lower_bound}
	\limsup_{k \rightarrow \infty} \| v_k \|_{L^{2 + \frac{4}{n}}} \geq m.
	\end{align}
Then there exists a sequence $\{ x_k \} \subseteq \mathbb{R}^n$ such that
(up to a subsequence)
	\begin{align*}
	v_k (\cdot + x_k) \rightharpoonup V \qquad
	\textup{and} \qquad
	\| V \|_{L^2} \geq \left( \frac{n}{n + 2} \right)^\frac{n}{4} \frac{m^{\frac{n}{2} + 1} + 1}{M^{\frac{n}{2}}} \| Q \|_{L^2}.
	\end{align*}
\end{theorem}

We will also use the following lemma from \cite{OhTo09}.

\begin{lemma}[\cite{OhTo09}]\label{lem:f_limit}
Let $T \in (0, \infty)$,
and assume that $f: [0, T) \rightarrow \mathbb{R}^+$ is a continuous function.
If $\displaystyle \lim_{t \rightarrow T} f(t) = \infty$,
then there exists a sequence $\{ t_k \}$ in $\subseteq [0, T)$ such that
	\begin{align*}
	t_k \rightarrow T \qquad
	\textup{and} \qquad
	\frac{\displaystyle \int_0^{t_k} f(\tau) \ d\tau}{f(t_k)} \rightarrow 0 \qquad
	\textup{as} \qquad
	k \rightarrow \infty.
	\end{align*}
\end{lemma}

Now we present the main result in this section.

\begin{theorem}[Concentration of mass]\label{thm:limiting_profile} 
Consider the Cauchy problem \eqref{eq:dnls_stark} with $p = 1 + \frac{4}{n}$.
Suppose that the solution $u$ of \eqref{eq:dnls_stark} blows up at finite time $T^* < \infty$,
i.e. $\| \nabla u(t) \|_{L^2} \rightarrow \infty$ as $t \rightarrow T^*$.
Then for any function $w(t)$ satisfying $w(t) \| \nabla u(t) \|_{L^2} \rightarrow \infty$ as $t \rightarrow T^*$,
there exists a function $x(t) \in \mathbb{R}^n$ such that
(up to a subsequence)
	\begin{align}
	\liminf_{t \rightarrow T^*} \| u(t) \|_{L^2 \left( |x - x(t)| < w(t) \right)}
	\geq \| Q \|_{L^2}.   \label{e:mass-concentration}
	\end{align}
\end{theorem}

\begin{proof}
Let $G(u)$ be the r.h.s. of the equation \eqref{eq:energy_decay},
i.e.
	\begin{align*}
	G(u)
	:= - 2 i \int_{\mathbb{R}^n} E \cdot u \nabla \overline{u} dx
	- 2 a \| \nabla u(t) \|_{L^2}^2
	+ 2 a \| u(t) \|_{L^{2 + \frac{4}{n}}}^{2 + \frac{4}{n}}.
	\end{align*}
Integrating \eqref{eq:energy_decay} on $[0, t)$,
we have
	\begin{align}\label{eq:E_G}
	\mathcal{E}_0(u(t))
	= \mathcal{E}_0(u_0)
	+ \int_0^t  G(u(\tau)) d\tau.
	\end{align}
Now we give an estimate of $G(u)$.
By H\"older's inequality and the mass decay \eqref{eq:mass_decay},
we have
	\begin{align*}
	\left| \int_{\mathbb{R}^n} E \cdot u \nabla \overline{u} dx \right|
	&\leq \| E \cdot u(t) \|_{L^2} \| \nabla u(t) \|_{L^2} \\
	&= |E| \, \| u(t) \|_{L^2} \| \nabla u(t) \|_{L^2}
	\leq |E| \, \| u_0 \|_{L^2} \| \nabla u(t) \|_{L^2},
	\end{align*}
and by Gagliardo-Nirenberg inequality and \eqref{eq:mass_decay},
we have
	\begin{align*}
	\left\| u(t) \right\|_{L^{2 + \frac{4}{n}}}^{2 + \frac{4}{n}}
	\leq C \left\| u(t) \right\|_{L^{2}}^{\frac{4}{n}} \left\| \nabla u(t) \right\|_{L^{2}}^{2}
	\leq C \left\| u_0 \right\|_{L^{2}}^{\frac{4}{n}} \left\| \nabla u(t) \right\|_{L^{2}}^{2},
	\end{align*}
so
	\begin{align}
	\left| G( u(t) ) \right|
	&\leq 2 |E| \, \| u_0 \|_{L^2} \| \nabla u(t) \|_{L^2}
	+ 2a \left\| \nabla u(t) \right\|_{L^2}^2
	+ 2a C \left\| u_0 \right\|_{L^{2}}^{\frac{4}{n}} \left\| \nabla u(t) \right\|_{L^2}^2 \notag \\
	&\lesssim \| \nabla u(t) \|_{L^2} + \| \nabla u(t) \|_{L^2}^2. \label{eq:G_estimate}
	\end{align}
Since $\left\| \nabla u(t) \right\|_{L^2} \rightarrow \infty$ as $t \rightarrow T^*$,
by Lemma \ref{lem:f_limit},
there exists a sequence $t_k \rightarrow T^*$ such that
	\begin{align*}
	\frac{\displaystyle \int_0^{t_k} \left( \| \nabla u(\tau) \|_{L^2} + \| \nabla u(\tau) \|_{L^2}^2 \right) d\tau}
	{\| \nabla u(t_k) \|_{L^2} + \| \nabla u(t_k) \|_{L^2}^2} \rightarrow 0\,,
	\end{align*}
so by \eqref{eq:G_estimate} we have
	\begin{align}\label{eq:integral_control}
	\frac{\displaystyle \int_0^{t_k} G( u(\tau) ) \ d\tau}{\left\| \nabla u(t_k) \right\|_{L^2}^2}
	\rightarrow 0.
	\end{align}
Let
	\begin{align*}
	\rho(t) := \frac{\| \nabla Q \|_{L^2}}{\| \nabla u(t) \|_{L^2}}\,, \qquad
	v(t, x) := \rho^{\frac{n}{2}} u(t, \rho x), \qquad
	\rho_k := \rho(t_k), \qquad
	v_k(x) := v(t_k, x).
	\end{align*}
Now we check that $\{ v_k \}$ defined above satisfies the assumptions in Theorem \ref{thm:compactness_lemma}.
Firstly,
by the mass decay \eqref{eq:mass_decay},
we have
	\begin{align*}
	\| v_k \|_{L^2}
	= \| u(t_k) \|_{L^2}
	\leq \| u_0 \|_{L^2},
	\end{align*}
showing the boundedness of $\{ v_k \}$.
Secondly,
since
	\begin{align*}
	\| \nabla v_k \|_{L^2}
	= \rho_k \| \nabla u(t_k) \|_{L^2}
	= \| \nabla Q \|_{L^2},
	\end{align*}
we know that inequality \eqref{eq:upper_bound} is satisfied with $M = \| \nabla Q \|_{L^2}$.
Finally,
by \eqref{eq:energy},
\eqref{eq:E_G},
and \eqref{eq:integral_control},
we have
	\begin{align*}
	\mathcal{E}_0(v_k)
	&= \| \nabla v_k \|_{L^2}^2
	- \frac{n}{n + 2} \| v_k \|_{L^{2 + \frac{4}{n}}}^{2 + \frac{4}{n}} \\
	&= \rho_k^2 \| \nabla u(t_k) \|_{L^2}^2
	- \frac{n}{n + 2} \left( \rho_k^2 \| u(t_k) \|_{L^{2 + \frac{4}{n}}}^{2 + \frac{4}{n}} \right) \\
	&= \rho_k^2 \, \mathcal{E}_0( u(t_k) ) \\
	&= \rho_k^2 \, \mathcal{E}_0(u_0)
	+ \rho_k^2 \, \int_0^{t_k}  G(u(\tau)) d\tau \\
	&= \frac{\| \nabla Q \|_{L^2}^2 \mathcal{E}_0(u_0)}{\| \nabla u(t_k) \|_{L^2}^2}
	+ \| \nabla Q \|_{L^2}^2  \frac{\displaystyle \int_0^{t_k}  G(u(\tau)) d\tau}{\| \nabla u(t_k) \|_{L^2}^2}
	\rightarrow 0 \qquad
	\textup{as } k \rightarrow \infty,
	\end{align*}
or equivalently,
	\begin{align*}
	\| v_k \|_{L^{2 + \frac{4}{n}}}^{2 + \frac{4}{n}}
	\rightarrow \frac{n + 2}{n} \| \nabla v_k \|_{L^2}^2
	= \frac{n + 2}{n} \| \nabla Q \|_{L^2}^2\,,
	\end{align*}
so inequality \eqref{eq:lower_bound} is also satisfied
with $m = \left( \frac{n + 2}{n} \| \nabla Q \|_{L^2}^2 \right)^\frac{n}{2n + 4}$.
Hence by Theorem \ref{thm:compactness_lemma},
there exists a sequence $\{ x_k \} \subseteq \mathbb{R}^n$ such that
	\begin{align}\label{eq:weak_convergence}
	\rho_k^\frac{n}{2} u(t_k, \rho_k \cdot + x_k) \rightharpoonup V
	\end{align}
weakly in $H^1$,
and
	\begin{align*}
	\| V \|_{L^2}
	&\geq \left( \frac{n}{n + 2} \right)^\frac{n}{4} \frac{m^{\frac{n}{2} + 1} + 1}{M^{\frac{n}{2}}} \| Q \|_{L^2} \\
	&= \left( \frac{n}{n + 2} \right)^\frac{n}{4}
	\frac{\left( \frac{n + 2}{n} \| \nabla Q \|_{L^2}^2 \right)^{\frac{n}{4}} + 1}{\| \nabla Q \|_{L^2}^{\frac{n}{2}}} \| Q \|_{L^2}
	\geq \| Q \|_{L^2}.
	\end{align*}
By \eqref{eq:weak_convergence},
for every $R > 0$,
there is
	\begin{align*}
	\liminf_{k \rightarrow \infty} \int_{|x| \leq R} \rho_k^n \left| u(t_k, \rho_k x + x_k) \right|^2 dx
	\geq \int_{|x| \leq R} \left| V \right|^2 dx,
	\end{align*}
or equivalently,
	\begin{align*}
	\liminf_{k \rightarrow \infty} \int_{|x - x_k| \leq \rho_k R} \left| u(t_k, x) \right|^2 dx
	\geq \int_{|x| \leq R} \left| V \right|^2 dx.
	\end{align*}

Now let $w(t)$ be a function satisfying $w(t) \| \nabla u(t) \|_{L^2} \rightarrow \infty$ as $t \rightarrow T^*$.
Then $\frac{w(t_k)}{\rho_k} \rightarrow \infty$ as $k \rightarrow \infty$,
so $\frac{w(t_k)}{\rho_k} \geq R$ for $k$ large enough.
Hence,
	\begin{align*}
	\liminf_{k \rightarrow \infty} \sup_{y \in \mathbb{R}^n} \int_{|x - y| \leq w(t_k)} \left| u(t_k, x) \right|^2 dx
	&\geq \liminf_{k \rightarrow \infty} \sup_{y \in \mathbb{R}^n} \int_{|x - y| \leq \rho_k R} \left| u(t_k, x) \right|^2 dx \\
	&\geq \liminf_{k \rightarrow \infty} \int_{|x - x_k| \leq \rho_k R} \left| u(t_k, x) \right|^2 dx \\
	&\geq \int_{|x| \leq R} \left| V \right|^2 dx,
	\end{align*}
and letting $R \rightarrow \infty$,
we have
	\begin{align*}
	\liminf_{k \rightarrow \infty} \sup_{y \in \mathbb{R}^n} \int_{|x - y| \leq w(t_k)} \left| u(t_k, x) \right|^2 dx
	\geq \int_{\mathbb{R}^n} \left| V \right|^2 dx
	\geq \| Q \|_{L^2}^2.
	\end{align*}
For every $t \in [0, T^*)$,
the function
	\begin{align*}
	y \mapsto \int_{|x - y| \leq w(t)} \left| u(t, x) \right|^2 dx
	\end{align*}
is continuous and vanishes as $|y| \rightarrow \infty$,
so the supremum is attained at some point $x(t) \in \mathbb{R}^n$.
Therefore,
	\begin{align*}
	\liminf_{k \rightarrow \infty} \int_{|x - x(t_k)| \leq w(t_k)} \left| u(t_k, x) \right|^2 dx
	= \liminf_{k \rightarrow \infty} \sup_{y \in \mathbb{R}^n} \int_{|x - y| \leq w(t_k)} \left| u(t_k, x) \right|^2 dx
	\geq \| Q \|_{L^2}^2\,,
	\end{align*}
and the proof of Theorem \ref{thm:limiting_profile} is complete.
\end{proof}

Now we give an alternative proof of Theorem \ref{thm:dnls_stark_threshold} (a) by virtue of the mass concentration property (\ref{e:mass-concentration})
in Theorem \ref{thm:limiting_profile}.

\begin{corollary}\label{cor:proof_a}
Theorem \ref{thm:dnls_stark_threshold} (a) holds for all dimensions.
\end{corollary}

\begin{proof} Since $u_0\in \cH^1$, according to Proposition \ref{p:dnls:VE}, 
there exist $0<T^*\le\iy$ and a unique solution $u(t)$ in $C([0,T^*), \cH^1)$. 
Assume  
$u$ blows up at $T^* < \infty$.  
Then by  (\ref{e:mass-concentration})
(taking $w(t) \equiv 1$),
we have
(up to a subsequence)
	\begin{align*}
	\liminf_{t \rightarrow T^*} \| u(t) \|_{L^2 \left( |x - x(t)| < 1 \right)}
	\geq \| Q \|_{L^2}
	\end{align*}
for some function $x(t) \in \mathbb{R}^n$.
However by \eqref{eq:mass_decay},
$\| u(t) \|_{L^2}$ decays in $t$,
so its limit inferior will be strictly less than $\| Q \|_{L^2}$,
contradictory to the prior inequality.
\end{proof}

\begin{remark}
By Theorem \ref{thm:limiting_profile},
if $\| u_0 \|_{L^2} > \| Q \|_{L^2}$ and $u$ blows up at finite time $T^*$,
then $T^* \leq \frac{1}{a} \log \left( \frac{\| u_0 \|_{L^2}}{\| Q \|_{L^2}} \right)$.
\end{remark}  

\bigskip  
\noindent 
{\bf Acknowledgment}  {S.Z. is partially supported by NNSFC grant No. 12071323 as Co-PI.}


\bigskip 
\begin{thebibliography}{99}

\bibitem{AkDouKaMcKi}
G.  Akrivis, V.  Dougalis, O.  Karakashian,  V.  McKinney,
Numerical approximation of singular solutions of the damped nonlinear Schr\"odinger equation,
{ENUMATH 97} (Heidelberg), World Scientific,
1998, 117-124.


\bibitem{AvHerbst}
J. E. Avron,  I. W. Herbst,
Spectral and scattering theory of Schr\"odinger operators related to the Stark effect,
{Comm. Math. Phys.} {52} (1977), no. 3, 239-254.





\bibitem{BaHuZh2019a} 
 N. Basharat, Y. Hu,  S.  Zheng, Blowup rate for mass critical rotational nonlinear Schr\"odinger equations,
  Contemporary Math. 725, 2019, 1-12.    

\bibitem{BHHZ23t}  
 N. Basharat, H. Hajaiej,  Y. Hu, S. Zheng,  Threshold for blowup and stability for NLS with rotation, 
  Annales Henri Poincar\'e   Vol.24, (2023), 1377-1416.  

\bibitem{BeDiaz23} P. B\'egout, J.I. D\'iaz, Finite time extinction for a critically damped Schr\"odinger equation with a sublinear nonlinearity,
 Advances in Differential Equations 28, (2023),  no. (3-4), 311-340.




\bibitem{CarNa}
R. Carles,  Y. Nakamura,
Nonlinear Schr\"odingere equations with Stark potential, 
{Hokkaido Math. J.} \textbf{33} (2004), 719-729.


\bibitem{Ca}
T. Cazenave,
\textit{Semilinear Schr\"odinger Equations},
Courant Lecture Notes in Mathematics {10},
American Mathematical Society,
Courant Institute of Mathematical Sciences, 2003.


\bibitem{CyFroeKirSi}
H. L. Cycon, R. G. Froese, W. Kirsch,  B. Simon,
\textit{Schr\"odinger operators with application to quantum mechanics and global geometry},
Springer Study Edition,
Texts and Monographs in Physics, Springer-Verlag, Berlin, 1987.

\bibitem{CuiHS17} J. Cui,  J. Hong,  L. Sun,   On global existence and blow-up for damped stochastic nonlinear Schr\"odinger equation.   Discrete and Continuous Dynamical Systems-series B 24 (2017), 6837.  

\bibitem{Dar}
M. Darwich,
Blowup for the damped $L^2$-critical nonlinear Schr\"odinger equation, 
{Adv. Differential Equations} {17} (2012), no. 3-4, 337-367.

\bibitem{DarwMolin16} M. Darwich, L. Molinet, Some remarks on the nonlinear Schr\"odinger equation with fractional dissipation,
J. Math. Phys. 1 October 2016; 57 (10): 101502. 




\bibitem{Dinh21b}
V. D. Dinh, Blow-up criteria for linearly damped nonlinear Schr\"odinger equations,
{Evol. Equ. Control Theory} \textbf{10} (2021), no. 3, 599-617.



\bibitem{FanSuZh22stochastic} 
C. Fan,  Y. Su,  D. Zhang,  A note on $\log$-$\log$ blow up solutions for stochastic nonlinear Schr\"odinger equations,
Stoch PDE:  Anal Comp 10, (2022), 1500-1514.  


\bibitem{Fib01damped}  G. Fibich,  Self-focusing in the damped nonlinear Schr\"odinger equation,  SIAM J. Appl. Math. 61(5), (2001), 1680-1705.


\bibitem{FibMeRa06} G. Fibich, F Merle, P Rapha\"el, Proof of a spectral property related to the singularity formation for the 
$L^2$ critical nonlinear Schr\"odinger equation,
Physica D: Nonlinear Phenomena 220 (1), (2006), 1-13.   


\bibitem{FKKV19} G. Fotopoulos, N.I. Karachalios, V. Koukouloyannis, K. Vetas, 
The linearly damped nonlinear Schr\"odinger equation with localized driving: spatiotemporal decay estimates and the emergence of extreme wave events, 
Z. Angew. Math. Phys. 71, 3 (2020), 1-23.




\bibitem{GraGre78}
S. Graffi and V. Grecchi,
Resonances in Stark effect and perturbation theory, 
{Commun. Math. Phys.} {62} (1978), 83-96.



\bibitem{Ha07}
E. Harrell II,
Perturbation theory and atomic resonances since Schr\"odinger's time,
\textit{Spectral theory and mathematical physics:
a Festschrift in honor of Barry Simon's 60th birthday}, 227-248,
Proc. Sympos. Pure Math., 76, Part 1, \textit{Amer. Math. Soc., Providence, RI}, 2007.

\bibitem{HmiKe}
T. Hmidi and S. Keraani,
Blowup theory for the critical nonlinear Schr\"odinger equations revisited,
{Int. Math. Res. Not.} {46} (2005), 2815-2828. 



\bibitem{KhKMTh10}  
C. Kharif,   R. Kraenkel,  M. Manna,   R. Thomas,  The modulational instability in deep water under the action of wind and dissipation,
 Journal of Fluid Mechanics, 664,  (2010), 138-149. 



\bibitem{Lau95} P. Lauren\c{c}ot, Long-time behaviour for weakly damped driven nonlinear Schr\"odinger equations in $\R^N$, $N\ge 3$,  
Nonlinear Differential Equations Appl. 2, (1995),  357-369. 



\bibitem{Malomed22mu} B. Malomed, Multidimensional dissipative solitons and solitary vortices,
 Chaos, Solitons and Fractals 163, (2022), 112526.




\bibitem{Mer93} F. Merle, Determination of blow-up solutions with minimal mass for nonlinear Schr\"odinger equations with critical power,
{Duke Math. J.} {69} (1993), no. 2, 427-454.   

\bibitem{MerRa04boot}  
F. Merle and P. Rapha\"el, On universality of blow-up profile for $L^2$   
critical nonlinear Schr\"odinger equation, Invent. Math., 156 (2004), 565-672.   

\bibitem{MerRa0}  
F. Merle, P. Rapha\"el,
Blow up dynamic and upper bound on the blow up rate for critical nonlinear Schr\"odinger equation,
{Ann. of Math.}  {161} (2005), no. 1, 157-222.


\bibitem{MerRa1}
F. Merle, P. Rapha\"el,
Profiles and quantization of the blow up mass for critical nonlinear Schr\"odinger equation,
{Comm. Math. Phys.} {253} (2005), no. 3, 675-704.


\bibitem{MeRa06sharp}   
F. Merle,  P. Rapha\"el. On a sharp lower bound on the blow-up rate for the $L^2$ critical 
 nonlinear Schr\"odinger equation, J. Amer. Math. Soc., 19 (2006):37-90.  



\bibitem{OhTo09}
M. Ohta,  G. Todorova,
Remarks on global existence and blowup for damped nonlinear Schr\"odinger equations,
{Discrete Contin. Dyn. Syst.} {23} (2009), no. 4, 1313-1325.



\bibitem{PePoVaz95}
V. Perez-Garcia, M. Porras,  L. Vazquez,
The nonlinear Schr\"odinger equation with dissipation and the moment method, 
{Phys. Lett. A} {202} (1995), 176-182. 

\bibitem{PlanRa07}  F. Planchon, P. Rapha\"el,   Existence and stability of the  $\log$-$\log$ blow-up dynamics for the $L^2$-critical nonlinear Schr\"odinger equation in a domain,
Ann. Henri Poincar\'e 8, (2007), 1177-1219. 


\bibitem{SunZh20}  
C.-M. Sun, J.-Q. Zheng,
Low regularity blowup solutions for the mass-critical NLS in higher dimensions,
Journal de Math\'ematiques Pures et Appliqu\'ees,
Vol. 134,
(2020),  255-298,



  

\bibitem{Ra05stabi}  P. Rapha\"el, Stability of the $\log$-$\log$ bound 
for blow-up solutions to the critical nonlinear  Schr\"odinger equation, Math.Ann., 331 (2005), 577-609. 

\bibitem{Tao}
T. Tao,
\textit{Nonlinear dispersive equations: Local and global analysis},
CBMS Regional Conference Series in Mathematics {106},
American Mathematical Society, Providence, RI, 2006.


\bibitem{TsuWa} T. Tsurumi,  M. Wadati,
Free fall of atomic laser beam with weak inter-atomic interaction,
{J. Phys. Soc. Jpn.} {70} (2001), 60-68.


\bibitem{Tsu}
M. Tsutsumi, 
Nonexistence of global solutions to the Cauchy problem for the damped nonlinear Schr\"odinger equations,
{SIAM J. Math. Anal.} {15} (1984), 357-366.


\bibitem{Wein}
M. Weinstein,
Nonlinear Schr\"odinger equations and sharp interpolation estimates,
{Comm. Math. Phys.} {87} (1983), no. 4, 567-576.




\bibitem{YangRouZh2018} K. Yang,  S. Roudenko,  Y.-X. Zhao, 
Blow-up dynamics and spectral property in the $L^2$-critical nonlinear Schr\"odinger equation in high dimensions,   
 Nonlinearity 31 (2018), no.9, 4354-4392.  



\bibitem{Z2012a} S. Zheng, Fractional regularity for nonlinear Schr\"odinger equations with magnetic fields. 
Contemp. Math 581, (2012), 271-285. 

\end{thebibliography}
\end{document}